\documentclass[11pt,a4paper]{article}
\usepackage[top=2.5cm,bottom=2.5cm,left=2.2cm,right=2.2cm]{geometry}% Add by K
\usepackage[T1]{fontenc}
\usepackage[utf8]{inputenc}
\usepackage{color}
\usepackage{graphicx}
\usepackage{amsfonts}
\usepackage{extarrows}
\usepackage{amsmath,amsthm,amssymb,color}
\usepackage{hyperref}
\usepackage{eepic}
\usepackage{lineno}
\usepackage{enumerate}	
\usepackage{paralist}
\usepackage{cite}
\usepackage{algorithm}
\usepackage{algorithmicx}
\usepackage{algpseudocode}
\usepackage{authblk}%作者的单位出现在名字下方，且不影响脚注
\usepackage{mathtools}%让文中公式命令flalign,equation可以实现
\usepackage{pifont}%让公式可以用带圈的数字进行标号
\usepackage[numbers,sort&compress]{natbib}%使引用格式[1,2,3,4,5]变成[1-5]

\usepackage{tikz}
\hypersetup
{
	colorlinks=true,
	linkcolor=blue,
	filecolor=blue,
	urlcolor=blue,
	citecolor=cyan,
}

\newtheorem{theorem}{Theorem}[section]

\newtheorem{lemma}[theorem]{Lemma}
\newtheorem{corollary}[theorem]{Corollary}
\newtheorem{conjecture}[theorem]{Conjecture}
\newtheorem{problem}[theorem]{Problem}

\theoremstyle{definition}

\newtheorem{claim}{\indent Claim}[theorem]      % 绑定到 section 计数器
\newtheorem{remark}[theorem]{Remark}

\newtheorem{case}{\indent Case}[section]

{\setlength{\leftmargini}{1.5\parindent}
 \begin{itemize}
 \setlength{\itemsep}{-1.1mm}}
{\end{itemize}}

\begin{document}
	%\pagewiselinenumbers
	\title{\bf Trisimplicial vertices in (fork, odd  parachute)-free graphs}
    \author[1]{\bf Kaiyang Lan \footnote{Email: kylan95@126.com.}}
\author[2]{\bf Feng Liu\footnote{Email: liufeng0609@126.com.}}
\author[3]{\bf Di Wu\footnote{Email: 1975335772@qq.com.}}
\author[4]{\bf Yidong Zhou\footnote{Email: zoed98@126.com.}}
\affil[1]{\footnotesize School of Mathematics and Statistics,
	Minnan Normal University, Zhangzhou, 363000, Fujian, China}
\affil[2]{\footnotesize Department of Mathematics,
	East China Normal University, Shanghai, 200241, China}
\affil[3]{\footnotesize Department of Mathematics and Physics,
	Nanjing Institute of Technology, Nanjing,  211167, Jiangsu, China}
\affil[4]{\footnotesize College of Computer Science,
	Nankai University, Tianjin, 300350, China. Email: zoed98@126.com.}

	\date{}%{\today}
	\maketitle
\begin{abstract}
An {\em odd hole} in a graph is an induced subgraph which is a cycle of odd length at least five.
An {\em odd parachute} is a graph obtained from an odd hole $H$ by adding a new edge $uv$ such
that $x$ is adjacent to $u$ but not to $v$ for each $x\in V(H)$. 
A graph $G$ is perfectly divisible if for each induced subgraph $H$ of $G$, $V(H)$ can be
partitioned into $A$ and $B$ such that $H[A]$ is perfect and $\omega(H[B])<\omega(H)$.
A vertex of a graph is {\em trisimplicial} if its neighbourhood is the union of three cliques.
In this paper, we prove that $\chi(G)\leq \binom{\omega(G)+1}{2}$ if $G$ is
a (fork, odd parachute)-free graph by showing that $G$ contains a trisimplicial vertex when $G$ is nonperfectly divisible.
This generalizes some results of Karthick, Kaufmann and Sivaraman [{\em Electron. J. Combin.} \textbf{29} (2022) \#P3.19], and Wu and Xu [{\em Discrete Math.} \textbf{347} (2024) 114121].
As a corollary, every nonperfectly divisible claw-free graph contains a trisimplicial vertex.

			
\end{abstract}
	
	%---------------------
\section{Introduction}
All graphs in this paper are finite and simple.
For a positive integer $k$, we use $[k]$ to denote the set $\{1,2,\ldots,k\}$.
Let $G$ be a graph.
A $k$-{\em coloring} of $G$ is a mapping $\varphi:V(G)\to[k]$ such that $\varphi(u)\neq\varphi(v)$ whenever $u$ and $v$ are adjacent in $G$.
The {\em chromatic number} of $G$
is the minimum integer $k$ such that $G$ admits a $k$-coloring.
For a vertex set $X\subseteq V(G)$, we use $G[X]$ to denote the subgraph of $G$ induced by $X$, and call $X$ a {\em clique} if $G[X]$ is a complete graph.
The {\em clique number} of $G$ is the maximum size taken over all cliques of $G$.
A graph $G$ is {\em perfect} if $\chi(H)=\omega(H)$ for each induced subgraph $H$ of $G$, and {\em imperfect} otherwise.
We use $\chi(G)$ and $\omega(G)$ to denote the chromatic number and clique number of $G$, respectively.
We use $P_k$ to denote a path on $k$ vertices.
 
We say that a graph $G$ {\em contains} a graph $H$ if some induced subgraph of $G$ is isomorphic to $H$.
A graph is $H$-{\em free} if it does not contain $H$.
When $\mathcal{H}$ is a set of graphs, $G$ is $\mathcal{H}$-{\em free} if $G$ contains no graph of $\mathcal{H}$.
A class of graphs $\mathcal{C}$ is {\em hereditary} if $G\in \mathcal{C}$ then all induced subgraphs of $G$ are also in $\mathcal{C}$.
A hereditary class of graphs $\mathcal{G}$ is called $\chi$-{\em bounded} if
there is a function $f$ (called a $\chi$-{\em binding function}) such that $\chi(G)\leq f(\omega(G))$ for every $G\in \mathcal{G}$.
If $f$ is additionally a polynomial function, then we say that $\mathcal{G}$ is polynomially $\chi$-{\em bounded}.
See \cite{BRIS2004,ISBR2019,AS20} for more results of this subject.

A {\em claw} is the complete bipartite graph $K_{1,3}$, see Figure \ref{Figure eight} for a depiction.
Since claw-free graphs are a generalization of line graphs, the class of claw-free graphs is widely studied in a variety of contexts and has a vast literature, see \cite{RF1997} for a survey.
It is natural to ask what properties of line graphs can be extended to all claw-free graphs.
A detailed and complete structural classification of claw-free graphs has been given by Chudnovsky and Seymour \cite{MCPS2008}.
Brause, Randerath, Schiermeyer and Vumar \cite{CBBR2019} proved that there is no linear $\chi$-binding function
even for a very special class of claw-free graphs.
Kim \cite{JHK95} showed that the Ramsey number $R(3,t)$ has order of magnitude $O(\frac{t^2}{\log t})$, and thus for any claw-free graph $G$, $\chi(G)\leq O(\frac{\omega(G)^2}{\log \omega(G)})$. 
Chudnovsky and Seymour \cite{MC2010} showed that every connected claw-free graph $G$ with a stable set of size at least three satisfies $\chi(G)\leq 2\omega(G)$.
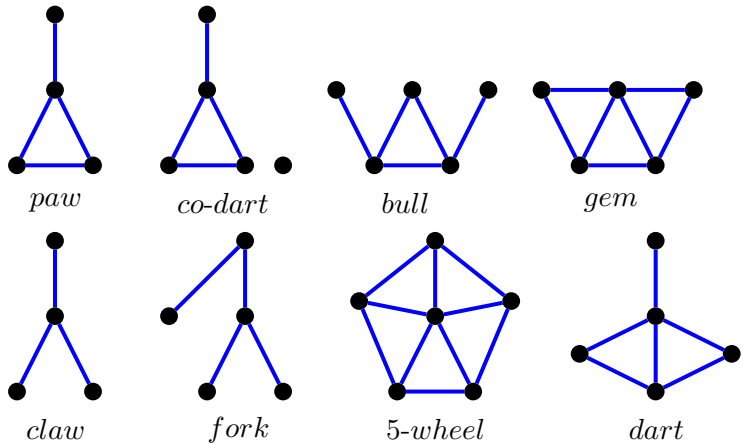
\begin{figure}[h]
	\begin{center}
		\begin{tikzpicture}%[auto,node distance=2.5cm, thick]
			
			\tikzset{std node fill/.style={draw=black, circle,fill=black, line width=1pt, inner sep=2pt}}
			
			\node[std node fill] (a1) at (0,0) {};
			%\node[xshift=6pt] at (v2.north)  {$v_1$};
			
			\node[std node fill] (a2) at (1,0) {};
			%\node[xshift=6pt] at (v2.east)  {$v_2$};
			
			\node[std node fill] (a3) at (0.5,1) {};
			%\node[xshift=6pt] at (v3.east)  {$v_3$};
			
			\node[std node fill] (a4) at (0.5,2) {};
			
			\node[xshift=6pt] at (0.3,-0.5)  {$paw$};
			
			\foreach \i/\j in {1/2,1/3,2/3,3/4}{\draw[blue,line width=1.5pt] (a\i)--(a\j);}
			
			%		\tikzset{std node fill/.style={draw=black, circle,fill=black, line width=1pt, inner sep=2pt}}
			
			\node[std node fill] (b1) at (2,0) {};
			%\node[xshift=6pt] at (v2.north)  {$v_1$};
			
			\node[std node fill] (b2) at (3,0) {};
			%\node[xshift=6pt] at (v2.east)  {$v_2$};
			
			\node[std node fill] (b3) at (2.5,1) {};
			%\node[xshift=6pt] at (v3.east)  {$v_3$};
			
			\node[std node fill] (b4) at (2.5,2) {};
			
			\node[std node fill] (b5) at (3.5,0) {};
			
			\node[xshift=6pt] at (2.5,-0.5)  {$co$-$dart$};
			
			\foreach \i/\j in {1/2,1/3,2/3,3/4}{\draw[blue,line width=1.5pt] (b\i)--(b\j);}

			\node[std node fill] (c1) at (4.7,0) {};
			%\node[xshift=6pt] at (v2.north)  {$v_1$};
			
			\node[std node fill] (c2) at (5.7,0) {};
			%\node[xshift=6pt] at (v2.east)  {$v_2$};
			
			\node[std node fill] (c3) at (4.2,1) {};
			%\node[xshift=6pt] at (v3.east)  {$v_3$};
			
			\node[std node fill] (c4) at (5.2,1) {};
			
			\node[std node fill] (c5) at (6.2,1) {};
			
			\node[xshift=6pt] at (4.9,-0.5)  {$bull$};
			
			\foreach \i/\j in {1/2,1/3,1/4,2/4,2/5}{\draw[blue,line width=1.5pt] (c\i)--(c\j);}

			\node[std node fill] (d1) at (7.4,0) {};
			%\node[xshift=6pt] at (v2.north)  {$v_1$};
			
			\node[std node fill] (d2) at (8.4,0) {};
			%\node[xshift=6pt] at (v2.east)  {$v_2$};
			
			\node[std node fill] (d3) at (6.9,1) {};
			%\node[xshift=6pt] at (v3.east)  {$v_3$};
			
			\node[std node fill] (d4) at (7.9,1) {};
			
			\node[std node fill] (d5) at (8.9,1) {};
			
			\node[xshift=6pt] at (7.6,-0.5)  {$gem$};
			
			\foreach \i/\j in {1/2,1/3,1/4,2/4,2/5,4/3,4/5}{\draw[blue,line width=1.5pt] (d\i)--(d\j);}
			
			\node[std node fill] (x1) at (7.4,-2.5) {};
			%\node[xshift=6pt] at (v2.north)  {$v_1$};
			
			\node[std node fill] (x2) at (8.4,-3) {};
			%\node[xshift=6pt] at (v2.east)  {$v_2$};
			
			\node[std node fill] (x3) at (8.4,-1) {};
			%\node[xshift=6pt] at (v3.east)  {$v_3$};
			
			\node[std node fill] (x4) at (8.4,-2) {};
			
			\node[std node fill] (x5) at (9.4,-2.5) {};
			
			\node[xshift=6pt] at (8.2,-3.5)  {$dart$};
			
			\foreach \i/\j in {1/2,1/4,2/4,2/5,4/3,4/5}{\draw[blue,line width=1.5pt] (x\i)--(x\j);}

			\node[std node fill] (e1) at (0,-3) {};
			%\node[xshift=6pt] at (v2.north)  {$v_1$};
			
			\node[std node fill] (e2) at (1,-3) {};
			%\node[xshift=6pt] at (v2.east)  {$v_2$};
			
			\node[std node fill] (e3) at (0.5,-2) {};
			%\node[xshift=6pt] at (v3.east)  {$v_3$};
			
			\node[std node fill] (e4) at (0.5,-1) {};
			
			\node[xshift=6pt] at (0.3,-3.5)  {$claw$};
			
			\foreach \i/\j in {1/3,2/3,3/4}{\draw[blue,line width=1.5pt] (e\i)--(e\j);}

			\node[std node fill] (f1) at (2.5,-3) {};
			%\node[xshift=6pt] at (v2.north)  {$v_1$};
			
			\node[std node fill] (f2) at (3.5,-3) {};
			%\node[xshift=6pt] at (v2.east)  {$v_2$};
			
			\node[std node fill] (f3) at (3,-2) {};
			%\node[xshift=6pt] at (v3.east)  {$v_3$};
			
			\node[std node fill] (f4) at (3,-1) {};
			
			\node[std node fill] (f5) at (2,-2) {};
			
			\node[xshift=6pt] at (2.7,-3.5)  {$fork$};
			
			\foreach \i/\j in {1/3,2/3,3/4,4/5}{\draw[blue,line width=1.5pt] (f\i)--(f\j);}

			\node[std node fill] (g1) at (5,-3) {};
			%\node[xshift=6pt] at (v2.north)  {$v_1$};
			
			\node[std node fill] (g2) at (6,-3) {};
			%\node[xshift=6pt] at (v2.east)  {$v_2$};
			
			\node[std node fill] (g3) at (4.5,-1.8) {};
			%\node[xshift=6pt] at (v3.east)  {$v_3$};
			
			\node[std node fill] (g4) at (5.5,-2) {};
			
			\node[std node fill] (g5) at (6.5,-1.8) {};
			
			\node[std node fill] (g6) at (5.5,-1) {};
			
			\node[xshift=6pt] at (5.3,-3.5)  {5-$wheel$};
			
			\foreach \i/\j in {1/2,1/3,1/4,2/4,2/5,4/3,4/5,4/6,3/6,5/6}{\draw[blue,line width=1.5pt] (g\i)--(g\j);}

			%	\node[xshift=6pt] at (6.7,-0.6)  {$hammer$};

		\end{tikzpicture}
		\caption{\small {Illustrations of the configurations.}}\label{Figure eight}
	\end{center}
\end{figure}
In general, Chudnovsky and Seymour \cite{MC2010} proved the following upper bound on the chromatic number for claw-free graphs.

\begin{theorem}[Chudnovsky-Seymour \cite{MC2010}]\label{claww2}
	Every claw-free graph $G$ satisfies $\chi(G)\leq \omega(G)^2$. Moreover, the bound is asymptotically tight.
\end{theorem} 

A graph is {\em perfectly divisible} \cite{CTH2018} if for each of its induced subgraph $H$, $V(H)$ can be partitioned into two sets $A$ and $B$ such that $H[A]$ is perfect and $\omega(H[B])<\omega(H)$.
See \cite{RC2023,MC2021,MC2019,CTH2018,WD2022,TKJK2022,DW2023} for more results of this flavour.
Note that by a simple induction on $\omega(G)$ we have that $\chi(G)\leq \binom{\omega(G)+1}{2}$ for each perfectly divisible graph.
A graph is {\em nonperfectly divisible} if it is not perfectly divisible.
A {\em fork} is the graph obtained from claw by subdividing an edge once, see Figure \ref{Figure eight} for a depiction.
A fork graph is also known as a chair graph.
The class of fork-free graphs is an immediate superclass of claw-free graphs and $P_4$-free
graphs, both of which are polynomially $\chi$-bounded.
It is natural to see what properties of claw-free graphs are also enjoyed by fork-free graphs.
Sivaraman \cite{TKJK2022} proposed the following conjecture.

\begin{conjecture}[Sivaraman \cite{TKJK2022}]\label{forkconj}
	The class of fork-free graphs is perfectly divisible.
\end{conjecture}
Conjecture \ref{forkconj} is still open,
here we briefly list some results related to $(\text{fork},H)$-free graphs for some small graph $H$, and refer the readers to \cite{MC2021,TKJK2022,DW2023cr} for more information on Conjecture \ref{forkconj} and related problems.

Kierstead and Penrice \cite{HAK1994} showed that for any tree $T$ of radius two, the class of $T$-free
graphs is $\chi$-bounded.
As a corollary, we see that fork-free graphs are $\chi$-bounded.
In \cite{TKJK2022} and \cite{BRIS2004}, the authors asked that whether there exists a polynomial (in particular, quadratic) $\chi$-binding
function for fork-free graphs.
Very recently, Liu, Schroeder, Wang and Yu \cite{XL2023} answered the problem and proved that
$\chi(G) \leq 7\omega(G)^2$ for every fork-free graph $G$.
This is a significant progress on Conjecture \ref{forkconj} that motivates us to pay more attention to the perfect divisibility of fork-free graphs.

A {\em hole} in a graph is an induced cycle of length at least four.
A {\em wheel} is a graph obtained from a hole $H$ by adding a new vertex adjacent to every vertex of $H$.
An $i$-{\em wheel} is a wheel such that its hole has $i$ vertices.
An $i$-wheel is called an {\em odd wheel} if $i$ is odd.
A {\em paw} is a graph obtained from $K_{1,3}$ by adding an edge joining two of its leaves, a {\em co-dart} is the union of $P_1$ and a paw, a {\em bull} is a graph consisting of a triangle with two disjoint pendant edges, a {\em gem} is the graph that consists of an induced four-vertex path plus a vertex which is adjacent to all the vertices of that path, see Figure \ref{Figure eight} for a depiction.
A {\em balloon} is a graph obtained from a hole by identifying respectively two consecutive vertices with two leaves of $K_{1,3}$. 
An $i$-{\em balloon} is a balloon such that its hole has $i$ vertices.
An $i$-balloon is called an {\em odd balloon} if $i$ is odd, see Figure \ref{Figure parachute} for a depiction.

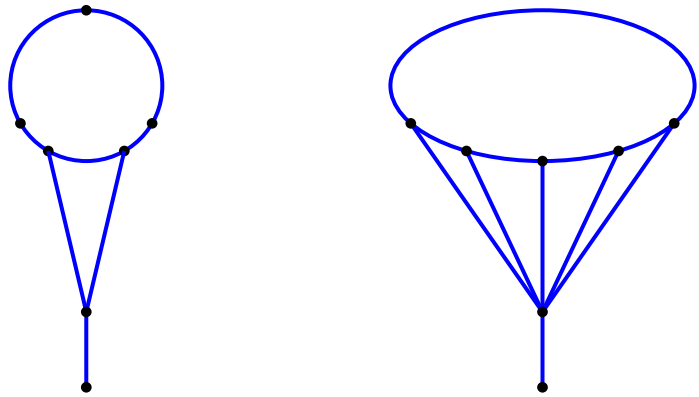
\begin{figure}[h]
	\begin{center}
\begin{tikzpicture}
	
	\draw[blue,line width=1.5pt] (-6,0) circle[radius=1];
	\foreach \angle in {210,240,90,300,330} {
	%	\draw[blue] ({2*cos(\angle)}, {sin(\angle)}) -- (-6,-3);black, circle,fill=black, line width=1pt, inner sep=2pt
		\fill ({cos(\angle)-6}, {sin(\angle)}) circle (2pt);
	}
	\foreach \angle in {240,300} {
		\draw[blue,line width=1.5pt] ({cos(\angle)-6}, {sin(\angle)}) -- (-6,-3);
	}
	\draw[blue,line width=1.5pt] (-6,-3) -- (-6,-4);
	\fill[] (-6,-3) circle (2pt);
	\fill[] (-6,-4) circle (2pt);

	\draw[blue,line width=1.5pt] (0,0) ellipse(2 and 1);
	\foreach \angle in {210,240,270,300,330} {
		\draw[blue,line width=1.5pt] ({2*cos(\angle)}, {sin(\angle)}) -- (0,-3);
		\fill ({2*cos(\angle)}, {sin(\angle)}) circle (2pt);
	}
	\draw[blue,line width=1.5pt] (0,-3) -- (0,-4);
	\fill[] (0,-3) circle (2pt);
	\fill[] (0,-4) circle (2pt);
\end{tikzpicture}
	\caption{\small {5-balloon and 5-parachute.}}\label{Figure parachute}
\end{center}
\end{figure}
Karthick, Kaufmann and Sivaraman \cite{TKJK2022} proved that Conjecture \ref{forkconj} is true on $(\text{fork}, H)$-free graphs when $H\in \{P_6,\text{co-dart, bull}\}$.
Recently, Wu and Xu \cite{DW2023} generalized these results in \cite{TKJK2022}, they proved that 

\begin{theorem}[Wu-Xu \cite{DW2023}]\label{ballon}
	The class of $($fork, odd balloon$)$-free graphs is perfectly divisible.
\end{theorem}

As an application, they also proved that

\begin{theorem}[Wu-Xu \cite{DW2023}]\label{forkgem}
	Every $(\text{fork, gem})$-free $G$ satisfies $\chi(G)\leq \binom{\omega(G)+1}{2}$.
\end{theorem}

A {\em parachute} is a graph obtained from a hole $H$ by adding an edge $uv$ such
that $u$ is complete to $V(H)$ and $v$ is anticomplete to $V(H)$. 
An $i$-{\em parachute} is a parachute such that its hole has $i$ vertices.
An $i$-parachute is called an {\em odd parachute} if $i$ is odd, see Figure \ref{Figure parachute} for a depiction.
In this paper, we consider the class of $(\text{fork}, \text{odd parachute})$-free graphs
which is an immediate superclass of both claw-free graphs and $(\text{fork}, \text{gem})$-free graphs.
A vertex of a graph is {\em bisimplicial} (resp., {\em trisimplicial}) if its neighborhood is the union of two (resp., three) cliques.
More precisely, we prove that
 \begin{theorem}\label{forkgem+tri}
	Let $G$ be a $(\text{fork}, \text{odd parachute})$-free graph.
	Then $G$ contains a trisimplicial vertex if $G$ is nonperfectly divisible.
\end{theorem}

As corollaries of Theorem \ref{forkgem+tri}, we have that

\begin{corollary}\label{MTHM1}
	Every nonperfectly divisible $\text{claw}$-free graph $G$ contains a trisimplicial vertex.
\end{corollary}

and improve Theorems \ref{claww2} and \ref{forkgem} by showing that

\begin{corollary}\label{MTHM}
	Every $(\text{fork}, \text{odd parachute})$-free graph $G$ satisfies $\chi(G)\leq \binom{\omega(G)+1}{2}$.
\end{corollary}
Let $G$ and $H$ be two vertex disjoint graphs. The {\em union} $G\cup H$ is the
graph with $V(G\cup H) =V(G)\cup (H)$ and $E(G\cup H) = E(G)\cup E(H)$.
The {\em join} $G +H$ is the graph with $V(G+ H)= V(G) + V(H)$ and $E(G + H) = E(G)\cup E(H)\cup \{xy: x\in V (G),y\in V(H)\}$.
A {\em dart} is the graph $P_1+(P_1\cup P_3)$, see Figure \ref{Figure eight} for a depiction.
The result in Corollary \ref{MTHM} also improves a result in \cite{TKJK2022}, where it was showed that $\chi(G)\leq\omega(G)^2$ if $G$ is $(\text{fork, dart})$-free.
We complete the proof of Theorem \ref{forkgem+tri} and Corollary \ref{MTHM} in Section \ref{secproof}.
In the rest of this section, we introduce additional notation and terminology.

Let $G$ be a graph and let $X,Y$ be disjoint subsets of $V(G)$.
For a vertex $v\in V(G)\setminus X$,
we say that $v$ is {\em complete} to $X$ in $G$ if $v$ is adjacent to every vertex in $X$; $v$ is {\em anticomplete} to $X$ in $G$ if $v$ is nonadjacent to every vertex in $X$.
We say that $X$ is {\em complete} (resp., {\em anticomplete}) to $Y$ if all vertices in $X$ are
complete (resp., anticomplete) to $Y$.
%We also say that $[X,Y]$ is both complete and anticomplete whenever $X$ or $Y$ is empty.
For any $v\in V(G)$, let $N_G(v)$ denote the set of all neighbors of $v$ in $G$, $d_G(v):=|N(v)|$, $N_G[v] := N_G(v)\cup \{v\}$, and let $M_G(v) := V(G)\setminus N_G[v]$.
The neighborhood $N_G(X)$ of $X$ is the set $\{v\in V(G)\backslash X:v$ is adjacent to a vertex of $X\}$.
Let $N_G[X]:=N_G(X)\cup X$ and $M_G(X) := V(G)\setminus N_G[X]$.
When there is no confusion, subscripts will be omitted.
For $u,v\in V(G)$, we simply write $u\sim v$ if $uv \in E(G)$, and write $u\nsim v$ if $uv \notin E(G)$.
For a vertex $u\in V(G)$, we will usually write $G\setminus u$ instead of $G[V(G)\setminus\{u\}]$.
\section{The main proof}\label{secproof}

The aim of this section is to prove Theorem \ref{forkgem+tri} and Corollary \ref{MTHM}.
Before our proof we list some useful results needed in later proofs. %we introduce additional notation and terminology, and 
Recall that a hole in a graph is
an induced cycle of length at least 4.
A hole is {\em odd} if it is of odd length.
The following result, proved by Wu and Xu \cite{DW2023}, is used several times in the sequel.

\begin{lemma}[Wu-Xu \cite{DW2023}]\label{CLMA1}
	Let $G$ be a $fork$-free graph, and $C:=v_1v_2\ldots v_nv_1$ be an odd hole contained in $G$. If there exist two adjacent vertices $u$ and $v$ in $V(G)\setminus V(C)$ such that $u$ is not anticomplete to $V(C)$ but $v$ is, then $N(u)\cap V(C)=\{v_j,v_{j+1}\}$, for some $j\in [n]$,  or $N(u)\cap V(C)=V(C)$.
\end{lemma}
	A vertex of a graph is {\em bisimplicial} if its neighborhood is the union of two cliques,
	and a vertex of a graph is {\em trisimplicial} if its neighbourhood is the union of three cliques.

	Let $B$ be an induced odd balloon in $G$, and $u\in V(B)$.
	We call $u$ be the {\em center} of $B$ if $d_B(u) = 3$ and $u$ has a neighbor of degree $1$ in $B$, and call $B$ a {\em minimum induced odd balloon} if $|V(B)|$ is minimum among all induced odd balloons of $G$. We are now ready to prove our main result.

\begin{proof}[\bf Proof of Theorem \ref{forkgem+tri}]
Let $G$ be a $(\text{fork}, \text{odd parachute})$-free graph that is nonperfectly divisible. In what follows, we will complete our proof by showing that $G$ contains a trisimplicial vertex.
		We prove this by contradiction.
		Suppose that $G$ contains no trisimplicial vertex.
		Since $G$ is nonperfectly divisible, $G$ must contain an odd balloon by Theorem \ref{ballon}.
		Let $B$ be a minimum induced odd balloon of $G$, and $u$ be the center of $B$.
		Let $C=v_1v_2\ldots v_nv_1$ be the odd hole of $B$ such that $n\geq 5$ and $n$ is odd, and $s$ be the vertex of degree $1$ of $B$.
		Without loss of generality, we may assume that $u$ is complete to $\{v_1,v_2\}$ and anticomplete to $\{v_3, v_4,\ldots, v_n\}$.  
\begin{claim}\label{vivj}
			Let $X$ denote the set of vertices which are complete to $\{v_i,v_{i+1}\}$ and anticomplete to $\{v_{i+2},v_{i+3}\}$ for some $i\in [n]$.
			Then $X$ is a clique.
		\end{claim}
Otherwise, suppose   that $X$ has two nonadjacent vertices $x_1$ and $x_2$. Then $\{v_{i+3},v_{i+2},v_{i+1},x_1,x_2\}$ induces a fork, a contradiction. This proves Claim \ref{vivj}. 

Let $T$ be the set of vertices in $N(u)\setminus \{v_1,v_2\}$ which are nonadjacent to $s$. 
\begin{claim}\label{ALIGN3}
$T$ is complete to $\{v_1,v_2\}$.
\end{claim}
Suppose to the contrary that there exists $t\in T$ is not complete to $\{v_1,v_2\}$. We may assume up to symmetry that $t\nsim v_2$. To forbid a fork on $\{s,t,u,v_2,v_3\}$, we have that $t\sim v_3$. If $t\sim v_n$, then $\{s,u,t,v_3,v_n\}$ induces a fork, a contradiction. So, $t\nsim v_n$. To forbid a fork on $\{s,t,u,v_1,v_n\}$, we have that $t\sim v_1$. If $t\sim v_{n-1}$, then $\{s,u,t,v_2,v_{n-1}\}$ induces a fork, a contradiction. So, $t\nsim v_{n-1}$. But now, $\{t,v_1,v_2,v_{n-1},v_n\}$ induces a fork, a contradiction. This proves Claim \ref{ALIGN3}.

Let $T_1=\{v\in T:N(v)\cap V(C)=\{v_1,v_2\}\}$, $T_2=\{v\in T:N(v)\cap V(C)=\{v_1,v_2,v_3\}\}$, and $T_3=\{v\in T:N(v)\cap V(C)=\{v_1,v_2,v_n\}\}$. By Claim \ref{vivj}, $T_1$ is a clique. Suppose that $T_2$ is not a clique. Let $t_3,t_4\in T_2$ be two nonadjacent vertices. Then $\{t_3,t_4,v_3,v_4,v_5\}$ induces a fork, a contradiction. Therefore, $T_2$ is a clique. Similarly, $T_3$ is a clique. 

\begin{claim}\label{CLM0}
$T_1\cup T_2\cup T_3$ is a clique.
\end{claim}
Suppose to the contrary that $t_1$ and $t_2$ are two nonadjacent vertices of $T_1\cup T_2\cup T_3$. Since $T_1$, $T_2$ and $T_3$ are all cliques by Claim \ref{ALIGN3}, we have that $t_1$ and $t_2$ belong to different elements of $\{T_1, T_2, T_3\}$. If $t_1\in T_2$ and $t_2 \in T_3$, then $\{t_2, t_1, u, s, v_n\}$ induces a fork, a contradiction. Otherwise, we assume by symmetry that $t_1 \in  T_1$ and $t_2 \in T_2$, then $\{t_1, t_2, v_1, v_n, v_{n-1}\}$ induces a fork, a contradiction again. This proves Claim \ref{CLM0}.

\begin{claim}\label{gem+Tclique}
$T$ is a clique.
\end{claim}
Suppose to the contrary that $t_1,t_2\in T$ are two nonadjacent vertices. Then $\{t_1,t_2\}$ is complete to $\{v_1,v_2\}$ by Claim \ref{ALIGN3}. Let $T^\prime=T\setminus(T_1\cup T_2\cup T_3)$. Since $T_1\cup T_2\cup T_3$ is a clique by Claim \ref{CLM0}, we have that $\{t_1,t_2\}\cap T^\prime\neq \emptyset$. Without loss of generality, we may assume that $t_1\in T^\prime$. If $t_1$ is complete to $\{v_3,v_n\}$, then $\{s,u,t_1,v_3,v_n\}$ induces a fork, a contradiction. So, $t_1$ is not complete to $\{v_3,v_n\}$. It follows that  $N(t_1)\cap \{v_4,\ldots,v_{n-1}\}\neq \emptyset$. Let $j\in \{4,\ldots, n-1\}$ be the smallest integer such that $t_1\sim v_j$. 
			
Suppose that $t_1$ is anticomplete to $\{v_3,v_n\}$. If $B$ is a $5$-balloon, then $v_j=v_4$. This implies that $\{u,t_1,v_4,v_5,v_3\}$ induces a fork, a contradiction. So, $n\geq 7$. To forbid a fork on $\{u,t_1,v_j,v_{j+1},v_{j-1}\}$, we have that $t_1\sim v_{j+1}$. Furthermore, if there exists a vertex $v_i\in \{v_{j+2},\ldots,v_{n-1}\}$ such that $t_1\sim v_i$, then $\{s,u,t_1,v_i,v_{j}\}$ induces a fork, a contradiction. Thus, $N(t_1)\cap \{v_3,\ldots, v_n\}=\{v_j,v_{j+1}\}$, which implies that 
\begin{flalign*}
\text{either} ~\{v_{j+1},\ldots ,v_n,v_1,t_1,u,s\}~ \text{or} ~\{v_2,\ldots,v_j,t_1,u,s\} 
\end{flalign*}
induces a smaller odd balloon as $n\geq 7$, a contradiction.
Therefore,
\begin{flalign}\label{x1}
N(t_1)\cap \{v_3,v_n\}\neq \emptyset. 
\end{flalign}
Note that $t_1$ is not complete to $\{v_3,v_n\}$. So, we may assume up to symmetry that $t_1\sim v_3$ and $t_1\nsim v_n$. We now show that 
\begin{flalign}\label{x2}
N(t_1)\cap V(C)=\{v_1,v_2,v_3,v_4\}. 
\end{flalign}
If $n=5$, then there is nothing to prove. So, we may assume that $n\geq 7$. If there exists a vertex $v_i\in \{v_{5},\ldots,v_{n-1}\}$ such that $t_1\sim v_i$, then $\{s,u,t_1,v_i,v_3\}$ induces a fork, a contradiction. This implies that $v_j=v_4$ and hence $N(t_1)\cap V(C)=\{v_1,v_2,v_3,v_4\}$. This proves (\ref{x2}).
			
If $t_2\in T_1\cup T_2$, then $\{v_4,t_1,u,s,t_2\}$ induces a fork, a contradiction. If $t_2\sim v_n$, then $\{v_n,t_2,u,s,t_1\}$ induces a fork, a contradiction.

Thus, $t_2\nsim v_n$ and hence $t_2\in T^\prime$. By a similar argument as $t_1$, it follows from (\ref{x1}) and (\ref{x2}) that $N(t_2)\cap V(C)=\{v_1,v_2,v_3,v_4\}$. If $n\geq 7$, then $\{t_1,t_2,v_4,v_5,v_6\}$ induces a fork, a contradiction. Therefore, $B$ is a $5$-balloon. 

By our assumption, $v_5$ is not a trisimplicial vertex. This implies that $N(v_5)\setminus \{v_1,v_4\}\neq \emptyset$. Let $r\in N(v_5)\setminus \{v_1,v_4\}$ be any vertex. 
			
Suppose that $r\sim s$. By Lemma \ref{CLMA1} and that $G$ is odd-parachute-free, we have that $N(r)\cap V(C)\in \{\{v_4,v_5\},\{v_5,v_1\}\}$. If $N(r)\cap \{t_1,t_2\}=\emptyset$, then $r\sim u$ as $\{u,s,r,t_1,t_2\}$ cannot induces a fork. This implies that $\{u,t_1,t_2,r,v_5\}$ induces a fork, a contradiction. Therefore, $N(r)\cap \{t_1,t_2\}\neq\emptyset$. By symmetry, we may assume that $r\sim t_1$. But this implies that $\{v_3,t_1,r,v_5,s\}$ induces a fork, a contradiction. Therefore, $r\nsim s$. By the arbitrariness of $r$, we have that 
\begin{flalign}\label{x3}
s ~\text{is anticomplete to}~ N(v_5). 
\end{flalign}
Let $T^u_5$ be the set of vertices in $N(v_5)\setminus \{v_1,v_4\}$ which are nonadjacent to $u$. Suppose that $T^u_5$ has two nonadjacent vertices $r_1$ and $r_2$. Note that $\{r_1,r_2\}$ is anticomplete to $s$ by (\ref{x3}). If $\{r_1,r_2\}$ is complete to $v_1$, then $\{v_1,r_1,r_2,u,s\}$ induces a fork, a contradiction. If $\{r_1,r_2\}$ is anticomplete to $v_1$, then $\{v_5,v_1,u,r_1,r_2\}$ induces a fork, a contradiction again. Therefore, we may assume up to symmetry that $r_1\nsim v_1$ and $r_2\sim v_1$. To forbid a fork on $\{r_1,v_5,v_1,t_1,t_2\}$, we have that $N(r_1)\cap  \{t_1,t_2\}\neq \emptyset$. Without loss of generality, we may assume that $r_1\sim t_1$. To forbid a fork on $\{r_1,u,s,t_1,t_2\}$, we have that $r_1\sim t_2$. Furthermore, to forbid a fork on $\{r_1,v_5,r_2,t_1,t_2\}$, we have that $N(r_2)\cap \{t_1,t_2\}\neq \emptyset$. By symmetry, we may assume that $t_1\sim r_2$. But now, $\{s,u,t_1,r_1,r_2\}$ induces a fork, a contradiction. This proves that $T^u_5$ is a clique.
			
Let $r'$ be any vertex in $T^u_5$. To forbid a fork on $\{r',v_5,v_4,v_1,u\}$, we have that either $r'\sim v_4$ or $r'\sim v_1$. Hence, $T_u^5\cup \{v_1,v_4\}$ is the union of two cliques.
Since $v_5$ is not a trisimplicial vertex, we have that $N(v_5)\setminus (T^u_5\cup \{v_1,v_4\})\neq \emptyset$. Let $r''$ be any vertex in $N(v_5)\setminus (T^u_5\cup \{v_1,v_4\})$. Then $r''\sim u$. If $r''\sim v_3$, then $\{s,u,r'',v_3,v_5\}$ induces a fork, a contradiction. By the arbitrariness of $r''$, all vertices in $N(v_5)\setminus (T^u_5\cup \{v_1,v_4\})$ are nonadjacent to $v_3$.

Since $v_5$ is not a trisimplicial vertex, we have that $N(v_5)\setminus (T^u_5\cup \{v_1,v_4\})$ is not a clique. Therefore, there exist two nonadjacent vertices $ r_3,r_4$ in $N(v_5)\setminus (T^u_5\cup \{v_1,v_4\})$. Note that $T=T_1\cup T_2\cup T_3\cup T'$ and $u$ is complete to $\{r_3,r_4\}$. Since $\{r_3,r_4\}$ is complete to $v_5$ and anticomplete to $s$ by (\ref{x3}), we have that $\{r_3,r_4\}\subseteq T'$. Then, by a similar argument as on $\{t_1,t_2\}$, we have that $N(r_3)\cap V(C)=N(r_4)\cap V(C)=\{v_1,v_2,v_4,v_5\}$. Again, by a similar argument as on $N(v_5)$, we can prove that $N(s)\cap N(v_3)=\emptyset$. It follows from Lemma \ref{CLMA1} that $N(s)\cap N(v_4)=\emptyset$. Therefore, we have that
\begin{flalign}\label{sv345}
\text{$N(s)\cap N(\{v_3,v_4,v_5\})=\emptyset$.}
\end{flalign}
Let $V_s$ denote the set of vertices in $N(s)$ which are anticomplete to $V(C)$. 
			
Suppose that there exists a vertex $s^\prime$ of $V_s$ such that $s^\prime\nsim u$. Then to forbid a fork on $\{u,s,t_1,t_2,s^\prime\}$, we have that $N(s')\cap \{t_1,t_2\}\neq \emptyset$. We may assume up to symmetry that $s^\prime \sim t_1$. But this implies that $\{s,s^\prime ,t_1,v_1,v_3\}$ induces a fork, a contradiction. By the arbitrariness of $s^\prime$, all vertices in $V_s$ is adjacent to $u$. Now, if there are two nonadjacent vertices $s_1,s_2$ in $V_s$, then $\{u,s_1,s_2,v_2,v_3\}$ induces a fork, a contradiction.
Therefore, 
\begin{flalign}\label{ALGNx0}
\text{$V_s$ is a clique.}
\end{flalign}
By the definition of $V_s$, it follows from Lemma \ref{CLMA1} and (\ref{sv345}) that 
\begin{flalign}\label{ALGN4}
\text{$N(s)\setminus V_s$ is complete to $\{v_1,v_2\}$.}
\end{flalign}
By our assumption, $s$ is not a trisimplicial vertex, it follows from (\ref{ALGNx0}) and (\ref{ALGN4}) that there are two nonadjacent vertices $s_3, s_4$ in $N(s)\setminus V_s$. But now we see that $\{v_4,v_3,v_2,s_3,s_4\}$ induces a fork, a contradiction. Therefore, $T$ is a clique. This proves Claim \ref{gem+Tclique}.

Let $D$ denote the set of vertices in $N(u)\setminus \{v_1,v_2\}$ which are anticomplete to $V(C)$. Clearly, $s\in D$. If $D$ has two nonadjacent vertices $d_1$ and $d_2$, then $\{d_1,d_2,u,v_2,v_3\}$ induces a fork, a contradiction. Therefore, 
\begin{flalign}\label{ALIGN2}
\text{$D$ is a clique.}
\end{flalign}
Let $T^+$ be the set of vertices in $N(u)\setminus (D\cup \{v_1,v_2\})$ which are not complete to $D$. We prove that
\begin{flalign}\label{T+clique}
\text{$T^+$ is a clique.}
\end{flalign}
Since each vertex $t^+$ of $T^+$ has a nonneighbor in $D$, say $d^+$, by replacing $s$ with $d^+$ in Claim \ref{ALIGN3}, we have that $t^+$ is complete to $\{v_1,v_2\}$. That is, $T^+$ is complete to $\{v_1,v_2\}$.

Suppose to the contrary that $T^+$ is not a clique. Let $t_3,t_4$ be two nonadjacent vertices in $T^+$. If $D$ has a vertex, say $d$, that is nonadjacent to both $t_3$ and $t_4$, then by replacing $s$ with $d$ in Claim \ref{gem+Tclique}, we have that $t_3\sim t_4$, a contradiction. Therefore, by the definition of $T^+$, there exist $d_3,d_4\in D$ such that $t_3\sim d_3, t_3\nsim d_4, t_4\sim d_4$ and $t_4\nsim d_3$. Furthermore, by replacing $s$ with $d_3$ in Claim \ref{ALIGN3}, we have that $t_4$ is complete to $\{v_1,v_2\}$. Similarly, $t_3$ is complete to $\{v_1,v_2\}$. By Lemma \ref{CLMA1} and that $G$ is odd-parachute-free, we have that $N(\{t_3,t_4\})\cap V(C)=\{v_1,v_2\}$. However, it follows that $\{v_4,v_3,v_2,t_3,t_4\}$ induces a fork, a contradiction. This proves (\ref{T+clique}).
		
Let $Q$ be the set of vertices in $N(u)\setminus(\{v_1,v_2\}\cup D)$ which are complete to $D$. Then, $N(u)=\{v_1,v_2\}\cup D\cup T^+\cup Q$. Let $Q_1$ be the set of vertices in $Q$ which are complete to $\{v_1,v_2\}$. Clearly, $Q_1$ is complete to $s$. By Lemma \ref{CLMA1} and that $G$ is odd-parachute-free, we have that $N(Q_1)\cap V(C)=\{v_1,v_2\}$. It follows from Claim \ref{vivj} that
\begin{flalign}\label{Q1clique}
\text{$Q_1$ is a clique.}
\end{flalign}
\begin{claim}\label{CLaimtu}
			$T^+\cup Q_1$ is a clique.
		\end{claim}
Suppose to the contrary that there exist two nonadjacent vertices $t$ and $q_1$ in $T^+\cup Q_1$. Since both $T^+$ and $Q_1$ are cliques by (\ref{T+clique}) and (\ref{Q1clique}), we may assume that $t\in T^+$ and $q_1\in Q_1$. By the definition of $T^+$, without loss of generality, we may assume that $t\nsim s$.
If $t$ is complete to $\{v_n,v_3\}$, then $\{t,v_n,v_3,u,s\}$ induces a fork, a contradiction.
So, by symmetry, we may assume that $t\nsim v_3$. Since $T^+$ is complete to $\{v_1,v_2\}$ and $N(Q_1)\cap V(C)=\{v_1,v_2\}$, we see that $\{s,q_1,v_2,v_3,t\}$ induces a fork, a contradiction.
This proves Claim \ref{CLaimtu}.

Since $T^+\cup Q_1$ is complete to $\{v_1,v_2\}$ and $ T^+\cup Q_1$ is a clique by Claim \ref{CLaimtu}, we have that
\begin{flalign}\label{T+Q1v12}
\text{$T^+\cup Q_1\cup \{v_1,v_2\}$ is a clique.}
\end{flalign}
		
Let $Q'=Q\setminus Q_1$. Then $N(u)=Q'\cup D\cup (T^+\cup Q_1\cup \{v_1,v_2\})$. We next prove that 
\begin{flalign}\label{qunion}
\text{$Q'$ is the union of two cliques.}
\end{flalign}
Since $u$ is not a trisimplicial vertex, $Q'$ is not a clique.
So, we may assume that there are two nonadjacent vertices $q_1,q_2$ in $Q'$.
Suppose that $\{q_1,q_2\}$ is anticomplete to $\{v_3,v_n\}$.
If $\{q_1,q_2\}$ is anticomplete to $\{v_1,v_2\}$, then $\{v_3,v_2,u,q_1,q_2\}$ induces a fork, a contradiction.
So, $\{q_1,q_2\}$ is not anticomplete to $\{v_1,v_2\}$.
By symmetry, we may assume that $q_2\sim v_2$.
Since $q_2\in Q'$, we have that $q_2\nsim v_1$.
Furthermore, since $q_2\sim s$, we have that $q_2\sim v_3$ by Lemma \ref{CLMA1}, contrary to our assumption.
Therefore, $\{q_1,q_2\}$ is not anticomplete to $\{v_3,v_n\}$. So, we may assume up to symmetry that $q_2\sim v_3$.

By Lemma \ref{CLMA1}, we have that $N(q_2)\cap V(C)\in\{\{v_2,v_3\},\{v_3,v_4\}\}$.
Furthermore, to forbid a fork on $\{v_n,v_1,u,q_1,q_2\}$, we have that $N(q_1)\cap \{v_1,v_n\}\neq \emptyset$. We consider two cases:

\begin{case}\label{CASE1}
$q_1\sim v_1.$
\end{case}
By the definition of $Q'$, we have that $q_1\nsim v_2$.
Furthermore, by Lemma \ref{CLMA1}, we have that $N(q_1)\cap V(C)=\{v_1,v_n\}$. To forbid a fork on $\{v_n,q_1,u,v_2,q_2\}$, we have that $q_2\sim v_2$. Then $N(q_2)\cap V(C)=\{v_2,v_3\}$. It follows that $\{v_4,v_3,v_2,q_2,s,q_1,v_1\}$ induces a $5$-balloon. 
By the minimality of $B$, we have that $B$ is a $5$-balloon. We define the following sets.
\begin{flalign}
Q_2:&=\{q\in Q':N(q)\cap V(C)=\{v_2,v_3\}\},\notag\\
Q_5:&=\{q\in Q':N(q)\cap V(C)=\{v_5,v_1\}\}.\notag
\end{flalign}
\begin{claim}\label{q1523}
Let $q',q''\in Q'$ with $q'\nsim q''$. If $q'\in Q_5$ (resp., $q'\in Q_2$), then $q''\in Q_2$ (resp., $q''\in Q_5$).
\end{claim}
Let $q',q''\in Q'$ with $q'\nsim q''$ and $q'\in Q_5$. To forbid a fork on $\{v_3,v_2,u,q',q''\}$, we have that $N(q'')\cap \{v_2,v_3\}\neq \emptyset$. By the definition of $Q'$,
it follows from Lemma \ref{CLMA1} that $N(q'')\cap V(C)\in\{\{v_2,v_3\},\{v_3,v_4\}\}$.
If $N(q'')\cap V(C)=\{v_3,v_4\}$, then $\{v_4,q'',u,q',v_2\}$ induces a fork, a contradiction.
So, $N(q'')\cap V(C)=\{v_2,v_3\}$.
That is, $q''\in Q_2$. The case that $q'\in Q_2$ is symmetric. This proves Claim \ref{q1523}.

By Claim \ref{vivj}, both $Q_2$ and $Q_5$ are cliques. Let $W=Q'\setminus(Q_2\cup Q_5)$. Suppose that $W\neq\emptyset$ and $w\in W$. Then $N(w)\cap V(C)\in\{\{v_3,v_4\},\{v_4,v_5\}\}$. By Claim \ref{q1523}, we have that $w$ is complete to $\{q_1,q_2\}$. By symmetry, we may assume that $N(w)\cap V(C)=\{v_3,v_4\}$.
But now $\{v_1,q_1,w,q_2,v_4\}$ induces a fork, a contradiction.
This proves that $W=\emptyset$. That is, $Q'$ is the union of two cliques. See Figure \ref{Figure F12} for a depiction.
\begin{case}\label{CASE2}
$q_1\nsim v_1$.
\end{case}
Since $N(q_1)\cap \{v_1,v_n\}\neq \emptyset$, we have that $N(q_1)\cap V(C)=\{v_{n-1}, v_n\}$ by Lemma \ref{CLMA1}.
If $n\geq 7$, then $\{v_{n-1},q_1,u,v_1,q_2\}$ induces an fork, a contradiction.
So, $n=5$, that is to say, $B$ is a $5$-balloon and $N(q_1)\cap V(C)=\{v_4,v_5\}$.
Note that $N(q_2)\cap V(C)\in\{\{v_2,v_3\},\{v_3,v_4\}\}$.
If $N(q_2)\cap V(C)=\{v_2,v_3\}$, then $\{v_{4},q_1,u,q_2,v_1\}$ induces an fork, a contradiction.
So, $N(q_2)\cap V(C)=\{v_3,v_4\}$.
We define the following sets.
\begin{flalign}
Q_3:&=\{q\in Q':N(q)\cap V(C)=\{v_3,v_4\}\},\notag\\
Q_4:&=\{q\in Q':N(q)\cap V(C)=\{v_4,v_5\}\}.\notag
\end{flalign}
\begin{claim}\label{q3445}
Let $q',q''\in Q'$ with $q'\nsim q''$. If $q'\in Q_4$ (resp., $q'\in Q_3$), then $q''\in Q_3$ (resp., $q''\in Q_4$).
\end{claim} 
Let $q',q''\in Q'$ with $q'\nsim q''$ and $q'\in Q_4$.
To forbid a fork on $\{v_3,v_2,u,q',q''\}$, we have that $N(q'')\cap \{v_2,v_3\}\neq \emptyset$. By the definition of $Q'$, it follows from Lemma \ref{CLMA1} that $N(q'')\cap V(C)\in\{\{v_2,v_3\},\{v_3,v_4\}\}$. If $N(q'')\cap V(C)=\{v_2,v_3\}$, then $\{v_3,q'',u,q',v_1\}$ induces a fork, a contradiction. So, $N(q'')\cap V(C)=\{v_3,v_4\}$. That is, $q''\in Q_3$. The case that $q''\in Q_3$ is symmetric. This proves Claim \ref{q3445}.
	
By Claim \ref{vivj}, both $Q_3$ and $Q_4$ are cliques. Let $W=Q'\setminus(Q_3\cup Q_4)$. Suppose that $W\neq \emptyset$ and $w\in W$. Then $N(w)\cap V(C)\in\{\{v_2,v_3\},\{v_1,v_5\}\}$. By Claim \ref{q3445}, we have that $w$ is complete to $\{q_1,q_2\}$.
By symmetry, we may assume that $N(w)\cap V(C)=\{v_2,v_3\}$.
But now $\{v_5,q_1,w,q_2,v_2\}$ induces a fork, a contradiction.
This proves that $W=\emptyset$. That is, $Q'$ is the union of two cliques. See Figure \ref{Figure F12} for a depiction.

By combining Cases \ref{CASE1} and \ref{CASE2}, this proves (\ref{qunion}).

\begin{figure}[h]
			\begin{center}
				\begin{tikzpicture}%[auto,node distance=2.5cm, thick]
					
					\tikzset{std node fill/.style={draw=black, circle, line width=1pt, inner sep=2pt}}

					\node[std node fill] (g1) at (5,-3) {\(v_1\)};
					%\node[xshift=6pt] at (v2.north)  {$v_1$};
					
					\node[std node fill] (g2) at (6,-3) {\(v_2\)};
					%\node[xshift=6pt] at (v2.east)  {$v_2$};
					
					\node[std node fill] (g3) at (6.5,-2) {\(v_3\)};
					
					\node[std node fill] (g4) at (5.5,-1) {\(v_4\)};
					
					\node[std node fill] (g5) at (4.5,-2) {\(v_5\)};
					%\node[xshift=6pt] at (v3.east)  {$v_3$};
					
					\node[std node fill] (g6) at (5.5,-4.3) {\(u\)};
					
					\node[std node fill] (g7) at (5.5,-5.6) {\(s\)};
					
					\node[std node fill] (g8) at (3.5,-4) {\(Q_5\)};
					
					\node[std node fill] (g9) at (7.5,-4) {\(Q_2\)};
					
					\node[xshift=6pt] at (5.3,-6.2)  {$Q'=Q_2\cup Q_5$};
					
					\foreach \i/\j in {1/2,2/3,3/4,4/5,1/5,1/6,2/6,6/7,8/1,8/5,9/2,9/3,6/8,6/9,7/8,7/9}{\draw[blue,line width=1.5pt] (g\i)--(g\j);}

					\node[std node fill] (h1) at (11,-3) {\(v_1\)};
					%\node[xshift=6pt] at (v2.north)  {$v_1$};
					
					\node[std node fill] (h2) at (12,-3) {\(v_2\)};
					%\node[xshift=6pt] at (v2.east)  {$v_2$};
					
					\node[std node fill] (h3) at (12.5,-2) {\(v_3\)};
					
					\node[std node fill] (h4) at (11.5,-1) {\(v_4\)};
					
					\node[std node fill] (h5) at (10.5,-2) {\(v_5\)};
					%\node[xshift=6pt] at (v3.east)  {$v_3$};
					
					\node[std node fill] (h6) at (11.5,-4.3) {\(u\)};
					
					\node[std node fill] (h7) at (11.5,-5.6) {\(s\)};
					
					\node[std node fill] (h8) at (9.5,-2) {\(Q_4\)};
					
					\node[std node fill] (h9) at (13.5,-2) {\(Q_3\)};
					
					\node[xshift=6pt] at (11.3,-6.2)  {$Q'=Q_3\cup Q_4$};
					
					\foreach \i/\j in {1/2,2/3,3/4,4/5,1/5,1/6,2/6,6/7,8/4,8/5,9/4,9/3,6/8,6/9,7/8,7/9}{\draw[blue,line width=1.5pt] (h\i)--(h\j);}
					
				\end{tikzpicture}
				\caption{\small {Illustrations of $Q'=F_1\cup F_2$.}}\label{Figure F12}
			\end{center}
		\end{figure}
Now, by (\ref{qunion}) we may let $F_1,F_2$ be two cliques such that $Q'=F_1\cup F_2$.
		Since $Q'$ is complete to $D$, we have that $F_1\cup D$ is a clique.
		Set $X_1=T^+\cup Q_1\cup \{v_1,v_2\}$, $X_2=F_1\cup D$ and $X_3=F_2$.
		Then $N(u)=X_1\cup X_2\cup X_3$.
		By (\ref{T+Q1v12}), we have that $X_1$ is a clique, which implies that $u$ is a trisimplicial vertex, a contradiction.
		This final contradiction completes the proof of Theorem \ref{forkgem+tri}.
\end{proof}
At last, as a corollary of Theorem \ref{forkgem+tri}, we can easily prove Corollary \ref{MTHM}.
	
	\begin{proof}[\bf Proof of Corollary \ref{MTHM}]
		Let $G$ be a $(\text{fork}, \text{odd parachute})$-free graph.
		If $G$ is perfectly divisible, then by a simple induction on $\omega(G)$ we have that $\chi(G)\leq \binom{\omega(G)+1}{2}$.
		So, we may assume that $G$ is nonperfectly divisible.
		It follows from Theorem \ref{forkgem+tri} that $G$ contains a trisimplicial vertex $u$.
		For $\omega(G)\leq 3$, we deduce the statement from a result of Randerath \cite{BR1993} that $\chi(G)\le4$.
		For $\omega(G)\geq 4$,	It follows that we can take any $\chi(G\setminus u)$-coloring of $G\setminus u$ and extend it to a $(\cfrac{1}{2}\omega(G)^2+\cfrac{1}{2}\omega(G))$-coloring of $G$.
		This completes the proof of Corollary \ref{MTHM}.
	\end{proof}

\begin{remark}
		We use $3P_1$ to denote the union of 3 copies of $P_1$.
		It is well known from \cite{CBBR2019} and \cite{BMS19} that the class of $3P_1$-free graphs does not admit a linear $\chi$-binding function, and since the class of $(\text{fork}, \text{odd parachute})$-free graphs is a superclass of $3P_1$-free graphs, it follows that the class of $(\text{fork}, \text{odd parachute})$-free graphs does not admit a linear $\chi$-binding function.
\end{remark}
At the end of this paper, to describe the structure of fork-free graphs, we prove that each nonperfectly divisible fork-free graph with small cliques has a trisimplicial vertex.
	This proof shares some ideas with the proof of Theorem \ref{forkgem+tri}. 
However, for the sake of completeness, we give the proof here.

\begin{theorem}\label{forkthw3}
If $G$ is a nonperfectly divisible fork-free graph with $\omega(G)\leq 3$, then $G$ has a trisimplicial vertex.
\end{theorem}
\begin{proof}[\bf Proof of Theorem \ref{forkthw3}]
Let $G$ be a nonperfectly divisible fork-free graph with $\omega(G)\leq 3$. If $G$ contains no odd balloon, then $G$ is perfectly divisible by Theorem \ref{ballon}. Therefore, we may assume that $G$ contains an odd balloon. By the definition of odd balloon, $\omega(G)=3$. Let $B$ be an induced odd balloon of $G$, and $u$ be the center of $B$. Let $C=v_1v_2\ldots v_nv_1$ be the odd hole of $B$ such that $n\geq 5$ and $n$ is odd, and $s$ be the vertex of degree $1$ of $B$. Without loss of generality, we may assume that $u$ is complete to $\{v_1,v_2\}$ and anticomplete to $\{v_3, v_4,\ldots, v_n\}$.
		
Let $D$ denote the set of vertices in $N(u)\setminus \{v_1,v_2\}$ which are anticomplete to $V(C)$. Clearly, $s\in D$. If $D$ has two nonadjacent vertices $d_1$ and $d_2$, then $\{v_3,v_2,u,d_1,d_2\}$ induces a fork, a contradiction. Therefore, 
\begin{flalign}\label{ALIGN22}
\text{$D$ is a clique.}
\end{flalign}
\begin{claim}\label{CCLAIM12}
$N(u)\setminus\{v_1,v_2,s\}$ is complete to $D$.
\end{claim}
Suppose to the contrary that there exists $t\in N(u)\setminus\{v_1,v_2,s\}$ such that $t$ is not complete to $D$. Without loss of generality, suppose $t\nsim s$.	Since $\omega(G)=3$, we have that $t$ is not complete to $\{v_1,v_2\}$. So, we may assume up to symmetry that $t\nsim v_2$. To forbid a fork on $\{s,t,u,v_2,v_3\}$, we have that $t\sim v_3$. If $t\sim v_n$, then $\{s,u,t,v_3,v_n\}$ induces a fork, a contradiction. So, $t\nsim v_n$. To forbid a fork on $\{s,t,u,v_1,v_n\}$, we have that $t\sim v_1$. If $t\sim v_{n-1}$, then $\{v_{n-1},t,u,s,v_2\}$ induces a fork, a contradiction. So, $t\nsim v_{n-1}$. But now, $\{v_{n-1},v_n,v_1,v_2,t\}$ induces a fork, a contradiction. This proves Claim \ref{CCLAIM12}.

Since $\omega(G)=3$ and $D\cup \{u\}$ is a clique by (\ref{ALIGN22}), we have $|D|\leq 2$. If $|D|=2$, then $N(u)\setminus(D\cup \{v_1,v_2\})=\emptyset$ by Claim \ref{CCLAIM12}. This implies that $N(u)=D\cup \{v_1,v_2\}$ and hence $u$ is a trisimplicial vertex. Suppose now that $D=\{s\}$. Let $Q=N(u)\setminus\{v_1,v_2,s\}$. We now prove that 
\begin{flalign}\label{qunion2}
\text{$|Q|\leq 2$.}
\end{flalign}
		
Suppose to the contrary that  $|Q|\geq 3$.
By Claim \ref{CCLAIM12}, $Q$ is complete to $D$.
Since $\omega(G)=3$, $Q$ is an independent set.
Let $Q=\{q_1,q_2,\ldots,q_t\}$.
Since $\omega(G)=3$, each $q_i$ is not complete to $\{v_1,v_2\}$.
Suppose that $\{q_1,q_2\}$ is anticomplete to $\{v_3,v_n\}$.
If $\{q_1,q_2\}$ is anticomplete to $\{v_1,v_2\}$, then $\{v_3,v_2,u,q_1,q_2\}$ induces a fork, a contradiction.
So, $\{q_1,q_2\}$ is not anticomplete to $\{v_1,v_2\}$.
By symmetry, we may assume that $q_2\sim v_2$. 
Furthermore, since $q_2\sim s$ and $q_2\nsim v_1$, we have that $q_2\sim v_3$ by Lemma \ref{CLMA1}, contrary to our assumption.
Therefore, $\{q_1,q_2\}$ is not anticomplete to $\{v_3,v_n\}$.
We assume up to symmetry that $q_1\sim v_3$.  To forbid a fork on $\{v_3,q_1,u,q_2,q_3\}$, we have that $N(v_3)\cap \{q_2,q_3\}\neq \emptyset$.
By symmetry, we may assume that $v_3\sim q_2$. By Lemma \ref{CLMA1}, $\{q_1,q_2\}$ is anticomplete to $\{v_1,v_n\}$.
However, it follows that $\{u,q_1,q_2,v_1,v_n\}$ induces a fork, a contradiction. Thus, $|Q|\leq 2$. This proves (\ref{qunion2}).
		
By Claim \ref{CCLAIM12}, $Q$ is complete to $D$. It follows from (\ref{qunion2}) that $Q\cup D$ is the union of two cliques. Since $N(u)=\{s,v_1,v_2\}\cup Q$, we have that $u$ is a trisimplicial vertex on $G$. This completes the proof of Theorem \ref{forkthw3}.
\end{proof}
Based on Theorems \ref{forkgem+tri} and \ref{forkthw3}, and Corollary \ref{MTHM1}, the following question may be interesting.
	
\begin{problem}\label{forkconjw2}
Every fork-free graph is either perfectly divisible or has a trisimplicial vertex.
\end{problem}

\section*{Acknowledgement}  The first author was partially supported by the Youth Foundation of Fujian Province (Grant No. JZ240035), the Fujian Alliance of Mathematics (Grant No. 2025SXLMQN08), and the Minnan Normal University Foundation (Grant No. KJ2023002). The second author was partially supported by grant from the National Natural Sciences Foundation of China (No.12271170) and Science and Technology
Commission of Shanghai Municipality (STCSM) grant 22DZ2229014.
\section*{Declaration}

	\noindent$\textbf{Conflict~of~interest}$
	The authors declare that they have no known competing financial interests or personal relationships that could have appeared to influence the work reported in this paper.
	
	\noindent$\textbf{Data~availability}$
	Data sharing not applicable to this paper as no datasets were generated or analysed during the current study.

\end{document}